\documentclass[a4paper,12pt]{article}
\usepackage{tikz}
\usetikzlibrary{matrix,arrows,decorations.markings}
\usepackage{amsmath,amsfonts,amssymb,amsthm,url,textcomp}
\usepackage{csquotes}
\usepackage{a4wide}
\usepackage[numbers]{natbib}
\usepackage{fancyhdr}
\usepackage{anyfontsize}
\usepackage{hyperref}
\usepackage{hhline}
\usepackage{leftidx}
\usepackage{mdframed}
\usepackage{enumitem}
\usepackage{float}

\allowdisplaybreaks

\newtheorem{theorem}{Theorem}\numberwithin{theorem}{section}
\newtheorem{definition}[theorem]{Definition}
\newtheorem{lemma}[theorem]{Lemma}

\newtheorem{question}[theorem]{Question}

\numberwithin{theoremm}{subsection}

\numberwithin{theoremmm}{subsubsection}

\theoremstyle{remark}

\newcommand{\Rad}{\operatorname{Rad}}
\newcommand{\Aut}{\operatorname{Aut}}
\newcommand{\Alt}{\operatorname{Alt}}
\newcommand{\PSL}{\operatorname{PSL}}

\newcommand{\ord}{\operatorname{ord}}
\newcommand{\Sym}{\operatorname{Sym}}

\newcommand{\Inn}{\operatorname{Inn}}
\newcommand{\id}{\operatorname{id}}

\newcommand{\M}{\operatorname{M}}
\newcommand{\J}{\operatorname{J}}
\newcommand{\PGL}{\operatorname{PGL}}

\newcommand{\GL}{\operatorname{GL}}

\newcommand{\Mod}[1]{\ (\textup{mod}\ #1)}

\newcommand{\B}{\operatorname{B}}

\newcommand{\IN}{\mathbb{N}}

\newcommand{\IF}{\mathbb{F}}

\newcommand{\IZ}{\mathbb{Z}}

\newcommand{\Suz}{\operatorname{Suz}}

\newcommand{\T}{\operatorname{T}}

\newcommand{\Tcal}{\mathcal{T}}

\newcommand{\Kcal}{\mathcal{K}}
\newcommand{\HS}{\operatorname{HS}}

\newcommand{\McL}{\operatorname{McL}}
\newcommand{\He}{\operatorname{He}}
\newcommand{\HN}{\operatorname{HN}}

\newcommand{\Fi}{\operatorname{Fi}}
\newcommand{\ON}{\operatorname{O'N}}

\newcommand{\SL}{\operatorname{SL}}

\begin{document}

\title{Finite groups with an automorphism that is a complete mapping}

\author{Alexander Bors\thanks{Johann Radon Institute for Computational and Applied Mathematics (RICAM), Altenberger Stra{\ss}e 69, 4040 Linz, Austria. \newline E-mail: \href{mailto:alexander.bors@ricam.oeaw.ac.at}{alexander.bors@ricam.oeaw.ac.at} \newline The author is supported by the Austrian Science Fund (FWF), project J4072-N32 \enquote{Affine maps on finite groups}. \newline 2010 \emph{Mathematics Subject Classification}: Primary: 20D45. Secondary: 20D05. \newline \emph{Key words and phrases:} Finite groups, Group automorphisms, Complete mappings.}}

\date{\today}

\maketitle

\abstract{We show that a finite group $G$ admitting an automorphism $\alpha$ such that the function $G\rightarrow G$, $g\mapsto g\alpha(g)$, is bijective is necessarily solvable.}

\section{Introduction}\label{sec1}

Let $G$ be a group. Following Evans \cite[Introduction]{Eva15a}, we call a bijective function $f:G\rightarrow G$
\begin{itemize}
\item a \emph{complete mapping} if and only if the function $G\rightarrow G$, $x\mapsto xf(x)$, is bijective.
\item an \emph{orthomorphism} if and only if the function $G\rightarrow G$, $x\mapsto x^{-1}f(x)$, is bijective.
\item a \emph{strong complete mapping} if and only if $f$ is both a complete mapping and an orthomorphism.
\end{itemize}

The study of complete mappings on groups has a long and rich history. Motivated by the construction of mutually orthogonal Latin squares, complete mappings were introduced by Mann in 1944 \cite{Man44a}. In 1950, Bateman showed that any infinite group admits complete mappings \cite{Bat50a}. For finite groups, the existence problem for complete mappings was picked up by Hall and Paige, who proved in their 1955 paper \cite{HP55a} that a finite group does \emph{not} admit complete mappings if its Sylow $2$-subgroup is cyclic and nontrivial. As we know today, the converse of Hall and Paige's result is true as well, but this was only proved in 2009 by Wilcox \cite{Wil09a}, Evans \cite{Eva09a} and Bray, whose contribution (dealing with the Janko group $\J_4$) only appeared in print later as part of the multi-author paper \cite{BCCSZ19a}.

Variants of the notion of complete mappings have been studied as well. For an overview of what is known about the above defined strong complete mappings, see e.g.~Subsection 2.6 in the survey paper \cite{Eva15a}, or Section 8.5 in the more recent book \cite{Eva18a}. In a field-theoretic context, Winterhof studied so-called \emph{$\Kcal$-complete mappings}, for which the expression assumed to define a bijection on the underlying structure can be more complicated and involve iterates \cite{Win14a}.

From a practical point of view, we mention the usefulness of complete mappings in check digit systems, where they are used to detect twin errors, see e.g.~\cite[Introduction]{SW10a}. In an abelian setting, orthomorphisms serve the analogous purpose of detecting adjacent transpositions, but, as noted by Schulz in \cite[Introduction]{Sch00a}, in a nonabelian setting, one needs to use so-called anti-symmetric mappings (see \cite[Definition 1.1]{Sch00a}) for this purpose instead.

In this paper, we will be concerned with complete mappings on finite groups that are group automorphisms. This also has a connection with another classical research topic in the theory of finite groups, namely fixed-point-free automorphisms (i.e., group automorphisms whose only fixed point is the identity element). Let us explain this in more detail. Consider the following definition:

\begin{definition}\label{kCompleteDef}
Let $G$ be a group, and let $k\in\IZ$. A bijective function $f:G\rightarrow G$ is called \emph{$k$-complete} if and only if the function $G\rightarrow G,x\mapsto x^kf(x)$, is also bijective.
\end{definition}

Then $1$-complete mappings are complete mappings in the above defined sense, and $(-1)$-complete mappings are orthomorphisms. Note that, as is well-known (and see e.g.~\cite[Corollary 8.66]{Eva18a}), an automorphism of a \emph{finite} group is $(-1)$-complete if and only if it is fixed-point-free. Hence all results on fixed-point-free automorphisms of finite groups can be viewed as results on $(-1)$-complete automorphisms. Consider, for example, the following, which is part of a celebrated result of Rowley:

\begin{theorem}\label{rowleyTheo}(Rowley, \cite[Theorem]{Row95a})
A finite group admitting a fixed-point-free (equivalently, a $(-1)$-complete) automorphism is solvable.
\end{theorem}

Henceforth, we will refer to Theorem \ref{rowleyTheo} as \enquote{Rowley's theorem}. The main result of this paper is the following analogue of Rowley's theorem for $1$-complete automorphisms:

\begin{theorem}\label{mainTheo}
A finite group admitting a $1$-complete automorphism (i.e., an automorphism that is also a complete mapping) is solvable.
\end{theorem}

We make the following four observations concerning Theorem \ref{mainTheo}:
\begin{enumerate}
\item The fact that the existence of a $1$-complete automorphism implies solvability is in stark contrast to the existence of general ($1$-)complete mappings. Indeed, recall that by the converse of Hall and Paige's theorem, any finite group with a nontrivial, noncyclic Sylow $2$-subgroup admits a complete mapping, and so, since the Sylow $2$-subgroups of nonabelian finite simple groups are noncyclic, every finite nonsolvable group admits a complete mapping (just not one that is also a group automorphism).
\item Theorem \ref{mainTheo} is not true in general for infinite groups. For example, the identity automorphism of a Tarski monster group is $1$-complete.
\item The identity automorphism of a finite group $G$ is $1$-complete if and only if $G$ is of odd order. Hence Theorem \ref{mainTheo} can also be viewed as a generalization of the Feit-Thompson theorem \cite{FT63a}.
\item As Rowley's theorem states that for finite groups, the existence of automorphisms that are orthomorphisms implies solvability, and our Theorem \ref{mainTheo} states that the existence of automorphisms that are complete mappings implies solvability, one may also ask whether a finite group admitting an automorphism that is an anti-symmetric mapping in the sense of Schulz \cite[Definition 1.1]{Sch00a} is necessarily solvable. The answer to this is also \enquote{yes}. Indeed, by \cite[Proposition 5.3]{Sch00a}, if $G$ is a finite group, then an automorphism $\alpha$ of $G$ is anti-symmetric if and only if $\alpha$ does not leave any nontrivial conjugacy class of $G$ invariant. But if $G$ is nonsolvable, then by Rowley's theorem, each automorphism $\alpha$ of $G$ has a nontrivial fixed point, say $x$, and so $\alpha$ leaves the nontrivial conjugacy class $x^G$ invariant, whence $\alpha$ is not anti-symmetric.
\end{enumerate}

Finally, we note that for endomorphisms $\varphi$ of general (not necessarily finite) groups $G$, the condition that the function $G\rightarrow G$, $g\mapsto g^{-1}\varphi(g)$ be surjective (such endomorphisms are called \emph{uniform}) has also been studied in several contexts and has some interesting applications. For example, the celebrated Lang-Steinberg theorem from the theory of algebraic groups states that if $G$ is a connected affine algebraic group over an algebraically closed field, then a surjective algebraic group endmorphism of $G$ with finite fixed point subgroup is uniform \cite[\S 10]{Ste68a}. Recently, in \cite{PS18a} and as an application of uniform automorphisms via factorisations of direct products, Praeger and Schneider gave a proof of the non-embeddability of quasiprimitive permutation groups of simple diagonal type into wreath products in product action. Their proof, in contrast to previous ones, works also for infinite groups and does not require the classification of finite simple groups.

Before we proceed with the next section, where we will discuss several lemmas that will facilitate the proof of Theorem \ref{mainTheo}, we introduce some notation used throughout this paper. We denote by $\IN^+$ the set of positive integers. The identity function on a set $X$ is denoted by $\id_X$, and the identity element of a group $G$ by $1_G$. The finite field with $q$ elements is denoted by $\IF_q$, and its multiplicative group of units by $\IF_q^{\ast}$. The automorphism group of a group $G$ will be denoted by $\Aut(G)$, and the inner automorphism group of $G$ by $\Inn(G)$. In this paper, the multiplication in $\Aut(G)$ is the usual function composition $\circ$, so that for $\alpha,\beta\in\Aut(G)$, the product $\alpha\beta$ denotes the automorphism that maps $g\in G$ to $\alpha(\beta(g))$. For a positive integer $m$, the symmetric and alternating groups of degree $m$ will be denoted by $\Sym(m)$ and $\Alt(m)$ respectively.

\section{Preparations for the proof of Theorem \ref{mainTheo}}\label{sec2}

Recall from the paragraph after Definition \ref{kCompleteDef} that an automorphism of a finite group is $(-1)$-complete (in the sense of Definition \ref{kCompleteDef}) if and only if it is fixed-point-free. In a similar vein, Evans gave an equivalent reformulation of $1$-completeness of finite group automorphisms (cf.~also \cite[Corollary 8.67]{Eva18a}):

\begin{lemma}\label{inversionLem}
Let $G$ be a finite group, and $\alpha$ an automorphism of $G$. The following are equivalent:
\begin{enumerate}
\item $\alpha$ is $1$-complete.
\item For all $\beta\in\alpha\Inn(G)\subseteq\Aut(G)$, the only element of $G$ inverted by $\beta$ is $1_G$.
\end{enumerate}
\end{lemma}

\begin{proof}
See \cite[Theorem 4.1(ii)]{Eva92a}.
\end{proof}

The following lemma is useful for structural arguments concerning finite groups that admit $k$-complete automorphisms:

\begin{lemma}\label{characteristicLem}
Let $G$ be a finite group, $N$ a characteristic subgroup of $G$, $k\in\IZ$, and $\alpha$ a $k$-complete automorphism of $G$. Then the following hold:
\begin{enumerate}
\item The automorphism $\tilde{\alpha}$ of $G/N$ induced by $\alpha$ is $k$-complete.
\item The restriction $\alpha_{\mid N}\in\Aut(N)$ is $k$-complete.
\end{enumerate}
\end{lemma}

\begin{proof}
For statement (1): Since $\alpha$ is $k$-complete, the function $G\rightarrow G$, $g\mapsto g^k\alpha(g)$, is surjective. Hence its \enquote{coarsification modulo $N$}, the function
\[
G/N\rightarrow G/N, gN\mapsto g^k\alpha(g)N=g^kN\cdot\alpha(g)N=(gN)^k\cdot\tilde{\alpha}(gN),
\]
is also surjective and thus (by the finiteness of $G/N$) bijective. But this just means by definition that $\tilde{\alpha}$ is $k$-complete, as we wanted to show.

For statement (2): Since $\alpha$ is $k$-complete, the function $G\rightarrow G$, $g\mapsto g^k\alpha(g)$, is injective. Hence its restriction to $N$, the function
\[
N\rightarrow N,n\mapsto n^k\alpha(n)=n^k\cdot(\alpha_{\mid N})(n),
\]
is also injective and thus (by the finiteness of $N$) bijective. But this just means by definition that $\alpha_{\mid N}$ is $k$-complete, as we wanted to show.
\end{proof}

Recall that our ultimate goal is to show that a finite nonsolvable group $G$ does not admit any $1$-complete automorphisms. An important special case is when $G$ is a nonabelian finite simple group, which will be dealt with in Lemma \ref{nfsgLem} below. But before that, we show the following, which will be useful for passing from nonabelian finite simple groups to finite nonsolvable groups in general (and it will also be used in the proof of Lemma \ref{nfsgLem}):

\begin{lemma}\label{powerLem}
Let $S$ be a nonabelian finite simple group, $n\in\IN^+$, and assume that $S$ does not admit any $1$-complete automorphisms. Then $S^n$ does not admit any $1$-complete automorphisms.
\end{lemma}

\begin{proof}
Choose a transversal $\Tcal$ for $\Inn(S)$ in $\Aut(S)$. Then since $\Aut(S^n)$ is the permutational wreath product $\Aut(S)\wr\Sym(n)$ (see e.g.~\cite[result 3.3.20, p.~90]{Rob96a}), an arbitrary coset of $\Inn(S^n)\cong\Inn(S)^n$ in $\Aut(S^n)$ looks like this:
\[
(\alpha_1,\ldots,\alpha_n)\sigma\Inn(S)^n=\Inn(S)^n(\alpha_1,\ldots,\alpha_n)\sigma,
\]
where $\alpha_1,\ldots,\alpha_n\in\Tcal$ and $\sigma\in\Sym(n)$ is a coordinate permutation on $S^n$. Let $(\beta_1,\ldots,\beta_n)\sigma$, where $\beta_i=\iota_i\alpha_i$ for some inner automorphism $\iota_i\in\Inn(S)$ and $i=1,\ldots,n$, be an arbitrary element of that coset. Our goal is to show that we can choose $\vec{\iota}=(\iota_1,\ldots,\iota_n)\in\Inn(S)^n$ such that $(\beta_1,\ldots,\beta_n)\sigma$ inverts some nontrivial element $(s_1,\ldots,s_n)\in S^n$, i.e., such that there is $\vec{s}=(s_1,\ldots,s_n)\in S^n\setminus\{(1_S,\ldots,1_S)\}$ with
\begin{equation}\label{eq2}
(\beta_1(s_{\sigma^{-1}(1)}),\beta_2(s_{\sigma^{-1}(2)}),\ldots,\beta_n(s_{\sigma^{-1}(n)})) = (s_1^{-1},s_2^{-1},\ldots,s_n^{-1}).
\end{equation}
To that end, we study the solution set in $S^n$ to Equation (\ref{eq2}), which we can equivalently rewrite into the following system of equations over $S$:
\begin{equation}\label{eq3}
\beta_i(s_{\sigma^{-1}(i)})=s_i^{-1}\text{ for }i=1,\ldots,n.
\end{equation}
Assume without loss of generality (relabeling indices if necessary) that $(1,2,\ldots,k)$ is a cycle of $\sigma$ for some $k\in\IN^+$. Then the equations in the system from Formula (\ref{eq3}) with $i>k$ only involve variables $s_j$ with $j>k$, and so those equations will be satisfied as long as we set $s_j:=1_S$ for $j>k$. We can thus focus on the first $k$ equations from Formula (\ref{eq3}), which look like this:
\begin{align*}
\beta_1(s_k) &=s_1^{-1}, \\
\beta_2(s_1) &=s_2^{-1}, \\
\beta_3(s_2) &=s_3^{-1}, \\
&\vdots \\
\beta_{k-1}(s_{k-2}) &=s_{k-1}^{-1}, \\
\beta_k(s_{k-1}) &=s_k^{-1},
\end{align*}
or equivalently
\begin{align*}
s_1 &=\beta_1(s_k)^{-1}, \\
s_2 &=\beta_2(s_1)^{-1}=(\beta_2\beta_1)(s_k), \\
s_3 &=\beta_3(s_2)^{-1}=(\beta_3\beta_2\beta_1)(s_k)^{-1}, \\
&\vdots \\
s_{k-1} &=\beta_{k-1}(s_{k-2})^{-1}=(\beta_{k-1}\beta_{k-2}\cdots \beta_1)(s_k)^{(-1)^{k-1}}, \\
s_k &=\beta_k(s_{k-1})^{-1}=(\beta_k\beta_{k-1}\cdots\beta_1)(s_k)^{(-1)^k}.
\end{align*}
We make a case distinction.
\begin{enumerate}
\item Case: $k$ is even. Then let $\vec{\iota}:=(\id_S,\ldots,\id_S)\in\Inn(S)^n$, so that $\beta_k\beta_{k-1}\cdots\beta_1=\alpha_k\alpha_{k-1}\cdots\alpha_1=:\gamma$. By Rowley's theorem \cite[Theorem]{Row95a}, we can choose an element $s_k\in S\setminus\{1_S\}$ with $\gamma(s_k)=s_k$, and by the above, if we set
\[
s_i:=(\alpha_i\alpha_{i-1}\cdots\alpha_1)(s_k)^{(-1)^i}
\]
for $i=1,2,\ldots,k-1$, then $(s_1,s_2,\ldots,s_k,1_S,\ldots,1_S)\in S^n$ is a nontrivial element of $S^n$ inverted by $(\alpha_1,\ldots,\alpha_n)\sigma=(\beta_1,\ldots,\beta_n)\sigma$, as required.
\item Case: $k$ is odd. Note: As $\vec{\iota}$ runs through all of $\Inn(S)^n$, the product
\[
\beta_k\beta_{k-1}\cdots\beta_1=(\iota_k\alpha_k)(\iota_{k-1}\alpha_{k-1})\cdots(\iota_1\alpha_1)
\]
runs through the entire coset $\alpha_k\alpha_{k-1}\cdots\alpha_1\Inn(S)$. It follows by assumption and Lemma \ref{inversionLem} that we can choose $\vec{\iota}\in\Inn(S)^n$ such that $\beta_k\beta_{k-1}\cdots\beta_1$ inverts some nontrivial element of $S$, say $s_k$. Then, by the above, if we set
\[
s_i:=(\beta_i\beta_{i-1}\cdots\beta_1)(s_k)^{(-1)^i}
\]
for $i=1,2,\ldots,k-1$, then $(s_1,s_2,\ldots,s_k,1_S,\ldots,1_S)\in S^n$ is a nontrivial element of $S^n$ inverted by $(\beta_1,\ldots,\beta_n)\sigma$, as required.
\end{enumerate}
\end{proof}

We are now ready to prove the following:

\begin{lemma}\label{nfsgLem}
Let $S$ be a nonabelian finite simple group. Then $S$ does not admit any $1$-complete automorphisms.
\end{lemma}

\begin{proof}
We use the classification of finite simple groups and go through several cases.
\begin{enumerate}
\item Case: $S$ is sporadic. Note: Since every nonabelian finite simple group contains an element of order $2$, the coset representative $\id_S$ of $\Inn(S)$ itself inverts some nontrivial element of $S$. Hence the assertion is clear if $S$ is a complete group, and we only need to consider the unique nontrivial coset of $\Inn(S)$ in the groups $\M_{12}$, $\M_{22}$, $\HS$, $\J_2$, $\McL$, $\Suz$, $\He$, $\HN$, $\Fi_{22}$, $\Fi_{24}'$, $\ON$, $\J_3$ and $\T$. Using the ATLAS of Finite Group Representations \cite{ATLAS}, one sees that for each of these groups $S$, the nontrivial coset $\Aut(S)\setminus\Inn(S)$ contains an element $\beta$ of order $2^k$ for some integer $k\geq2$ (actually, one may choose $k=2$ except for $S=\ON$, where one may choose $k=3$), so that $\beta$ centralizes and thus inverts the order $2$ element $\beta^{2^{k-1}}\in S$.
\item Case: $S=\Alt(m)$ for some $m\geq5$. First, assume that $m=6$. As explained in the argument for the sporadic groups case, we only need to worry about the three nontrivial cosets of $\Inn(S)$ in $\Aut(S)$, each of which contains an order $4$ automorphism $\beta$ by the ATLAS of Finite Group Representations \cite{ATLAS}, and $\beta$ centralizes and thus inverts the order $2$ element $\beta^2\in S$. So assume now that $m\not=6$. Then there is only one nontrivial coset of $\Inn(S)\cong\Alt(m)$ in $\Aut(S)\cong\Sym(m)$, which contains the order $4$ automorphism corresponding to the $4$-cycle $(1,2,3,4)$.
\item Case: $S$ is of Lie type. By Lemma \ref{inversionLem}, it suffices to show that each coset $\beta\Inn(S)$ of $\Inn(S)$ in $\Aut(S)$ contains an automorphism which inverts some nontrivial element of $S$. Assume without loss of generality that the coset representative $\beta$ is chosen in the product set $D_S\Phi_S\Gamma_S$, where $D_S$ and $\Phi_S$ denote groups of diagonal and field automorphisms of $S$ respectively, and $\Gamma_S$ denotes a set of graph automorphisms of $S$ (see \cite[Section 2.5]{GLS98a} for information on automorphisms of finite simple groups of Lie type). Then by \cite[proof of Proposition 7, pp.~868f.]{Nik16a}, there is a subgroup $S_1\leq S$ left invariant by $\beta$ and such that $S_1$ is isomorphic to one of the following:
\begin{itemize}
\item a quasisimple group of type $A_1$, i.e., a quotient of the form $\SL_2(q)/Z$, where $q\geq4$ is a prime power and $Z$ is a subgroup of the center $\zeta\SL_2(q)$;
\item the direct product $\PSL_2(q)\times\PSL_2(q)$ for some prime power $q\geq4$; or
\item the exceptional Suzuki group $\leftidx{^2}\B_2(q)$ with $q=2^{2m+1}$ for some $m\in\IN^+$.
\end{itemize}
Note that we are done if we can show that $\beta\Inn(S_1)$ contains an automorphism inverting some nontrivial element of $S_1$. Since $\beta$ normalizes $S_1$, it thus suffices to show that none of the groups listed above to one of which $S_1$ is isomorphic admit any $1$-complete automorphisms. By Lemmas \ref{characteristicLem} and \ref{powerLem}, this, in turn, is equivalent to the assertion that none of the groups $\PSL_2(q)$, where $q\geq4$ is a prime power, or $\leftidx{^2}B_2(2^{2m+1})$, where $m\in\IN^+$, admits a $1$-complete automorphism (note that in the quasisimple case, $Z$ is characteristic in $\SL_2(q)$, since it is a (characteristic) subgroup of the cyclic group $\zeta\SL_2(q)$, which is characteristic in $\SL_2(q)$). It is this last assertion that we will now show (using Lemma \ref{inversionLem} again) to conclude our proof of Lemma \ref{nfsgLem}.

As for the Suzuki groups $S=\leftidx{^2}B_2(2^{2m+1})$, note that any coset of $\Inn(S)$ in $\Aut(S)$ contains a field automorphism $\varphi\in\Phi_S$, which centralizes a copy of $\leftidx{^2}B_2(2)$, the Frobenius group of order $20$. In particular, $\varphi$ centralizes and thus inverts some order $2$ element of $S$, as required.

We now turn to the projective special linear groups $S=\PSL_2(q)$, with $q=p^f\geq4$ a prime power. In what follows, for $M\in\GL_2(q)$, we denote by $\overline{M}$ the image of $M$ under the canonical projection $\GL_2(q)\rightarrow\PGL_2(q)$. We make a subcase distinction:
\begin{enumerate}
\item Subcase: $p=2$. Then each coset of $\Inn(S)$ in $\Aut(S)$ contains a field automorphism $\varphi\in\Phi_S$, which centralizes a copy of $\PSL_2(2)\cong\Sym(3)$ in $S$, and so $\varphi$ centralizes and thus inverts some order $2$ element of $S$, as required.
\item Subcase: $p>2$. The cosets of $\Inn(S)$ in $\Aut(S)$ which contain a field automorphism can be handled as in the previous subcase, \enquote{$p=2$}. So, assume that we are considering a coset of the form
\[
\overline{\begin{pmatrix}\xi & 0 \\ 0 & 1\end{pmatrix}}\phi^i\Inn(S),
\]
where $\xi$ is a generator of $\IF_q^{\ast}$ and $\phi$ is the standard Frobenius automorphism
\[
\overline{\begin{pmatrix}a & b \\ c & d\end{pmatrix}} \mapsto \overline{\begin{pmatrix}a^p & b^p \\ c^p & d^p\end{pmatrix}}
\]
of $S$, and where $i\in\{0,1,\ldots,f-1\}$. We make a subsubcase distinction:
\begin{enumerate}
\item Subsubcase: $q\equiv 1\Mod{4}$. Using that the subgroups of $\IZ/f\IZ$ generated by $i+f\IZ$ and $\gcd(f,i)+f\IZ$ respectively are equal, we see that the coset representative
\[
\overline{\begin{pmatrix}\xi & 0 \\ 0 & 1\end{pmatrix}}\phi^i
\]
has the power
\begin{align*}
&\left(\overline{\begin{pmatrix}\xi & 0 \\ 0 & 1\end{pmatrix}}\phi^i\right)^{\frac{f}{\gcd(f,i)}\cdot\frac{p^{\gcd(f,i)}-1}{2}} \\
&= \left(\overline{\begin{pmatrix}\xi^{1+p^i+p^{2i}+\cdots+p^{\left(\frac{f}{\gcd(f,i)}-1\right)i}} & 0 \\ 0 & 1\end{pmatrix}}\right)^{\frac{p^{\gcd(f,i)}-1}{2}} \\
&= \left(\overline{\begin{pmatrix}\xi^{1+p^{\gcd(f,i)}+p^{2\gcd(f,i)}+\cdots+p^{(\frac{f}{\gcd(f,i)}-1)\gcd(f,i)}} & 0 \\ 0 & 1\end{pmatrix}}\right)^{\frac{p^{\gcd(f,i)}-1}{2}} \\
&= \overline{\begin{pmatrix}\xi^{\frac{p^f-1}{p^{\gcd(f,i)}-1}} & 0 \\ 0 & 1\end{pmatrix}}^{\frac{p^{\gcd(f,i)}-1}{2}} = \overline{\begin{pmatrix}\xi^{\frac{p^f-1}{2}} & 0 \\ 0 & 1\end{pmatrix}}=\overline{\begin{pmatrix}-1 & 0 \\ 0 & 1\end{pmatrix}},
\end{align*}
which is an order $2$ element of $S$ that is centralized and thus inverted by the coset representative
\[
\overline{\begin{pmatrix}\xi & 0 \\ 0 & 1\end{pmatrix}}\phi^i,
\]
as required.
\item Subsubcase: $q\equiv 3\Mod{4}$. Then $-1$ is a non-square in $\IF_q$, and so
\[
\overline{\begin{pmatrix}0 & 1 \\ 1 & 0\end{pmatrix}}\in\PGL_2(q)\setminus\PSL_2(q),
\]
whence we may choose the coset representative of the form
\[
\overline{\begin{pmatrix}0 & 1 \\ 1 & 0\end{pmatrix}}\phi^i
\]
instead. But this representative maps the order $2$ element
\[
s:=\overline{\begin{pmatrix}0 & 1 \\ -1 & 0\end{pmatrix}}\in\PSL_2(q)=S
\]
to
\[
\overline{\begin{pmatrix}0 & 1 \\ 1 & 0\end{pmatrix}}\cdot\overline{\begin{pmatrix}0 & 1 \\ -1 & 0\end{pmatrix}}\cdot\overline{\begin{pmatrix}0 & 1 \\ 1 & 0\end{pmatrix}}=\overline{\begin{pmatrix}0 & -1 \\ 1 & 0\end{pmatrix}}=\overline{\begin{pmatrix}0 & 1 \\ -1 & 0\end{pmatrix}}=s,
\]
so that $s$ is inverted by the coset representative
\[
\overline{\begin{pmatrix}0 & 1 \\ 1 & 0\end{pmatrix}}\phi^i,
\]
as required.
\end{enumerate}
\end{enumerate}
\end{enumerate}
\end{proof}

\section{Proof of Theorem \ref{mainTheo}}\label{sec3}

We will show that if $G$ is a finite nonsolvable group, then $G$ does not admit any $1$-complete automorphisms. Assume otherwise. Denote by $\Rad(G)$ the \emph{solvable radical of $G$}, the largest solvable normal subgroup of $G$. Since the class of finite solvable groups is closed under extensions, the quotient $G/\Rad(G)$ has no nontrivial solvable normal subgroups, and hence, by \cite[3.3.18, p.~89]{Rob96a}, the socle (i.e., the subgroup generated by all the minimal nontrivial normal subgroups) $T$ of $G/\Rad(G)$ is a direct product of nonabelian finite simple groups. In particular, $T$ has a characteristic subgroup of the form $S^n$ for some nonabelian finite simple group $S$ and some $n\in\IN^+$, and that subgroup is also characteristic in $G/\Rad(G)$. It follows by Lemma \ref{characteristicLem} that $S^n$ admits a $1$-complete automorphism. Hence, by Lemma \ref{powerLem}, $S$ admits a $1$-complete automorphism, which contradicts Lemma \ref{nfsgLem}.

\section{Concluding remarks}\label{sec4}

We conclude this paper with some related open questions for further research. Observe that if $G$ is a finite group and $k\in\IZ$ is such that $\gcd(k+1,|G|)=1$, then the identity automorphism of $G$ is $k$-complete. In particular, every finite group admits a $(-2)$-complete automorphism, and each of the Suzuki groups $\leftidx{^2}B_2(2^{2m+1})$ admits a $2$-complete automorphism. Therefore, the next larger positive value of $k$ for which it makes sense to ask whether an analogue of Theorem \ref{mainTheo} holds is $3$:

\begin{question}\label{ques1}
Are finite groups admitting a $3$-complete automorphism necessarily solvable?
\end{question}

Similarly, one can show that for all even integers $k$ with $|k|\leq 12$, there is a finite nonsolvable group $G$ admitting a $k$-complete automorphism, which raises the following question:

\begin{question}\label{ques2}
Is it true that for every even integer $k$, there is a finite nonsolvable group $G$ admitting a $k$-complete automorphism?
\end{question}

With our definition of a $k$-complete automorphism $\alpha$ of a group $G$, we are viewing the function $G\rightarrow G$, $g\mapsto g\alpha(g)$, as a special case of the family of functions $G\rightarrow G$, $g\mapsto g^k\alpha(g)$. A different possibility to generalize it is by considering the functions $G\rightarrow G$, $g\mapsto g\alpha(g)\alpha^2(g)\cdots\alpha^k(g)$ for $k\in\IN^+$. This is more in the spirit of Winterhof's $\Kcal$-complete mappings from \cite{Win14a} already mentioned in Section \ref{sec1}. In this context, the following question, asking about generalizations of Theorem \ref{mainTheo} in a different direction, is of interest:

\begin{question}\label{ques3}
For which positive integers $k$ is it the case that for all finite groups $G$, the existence of an automorphism $\alpha$ of $G$ such that the function $G\rightarrow G$, $g\mapsto g\alpha(g)\alpha^2(g)\cdots\alpha^k(g)$, is bijective implies that $G$ is solvable?
\end{question}

Our Theorem \ref{mainTheo} says that the solvability implication in Question \ref{ques3} holds for $k=1$, whereas, for example, it does not hold for $k=2$ (by letting $G=\leftidx{^2}B_2(2^{2m+1})$ and $\alpha=\id_G$). We note that if $k=\ord(\alpha)-1$, where $\ord(\alpha)$ denotes the order of the automorphism $\alpha$, then the contrary condition that $g\alpha(g)\alpha^2(g)\cdots\alpha^k(g)=1_G$ for all $g\in G$ means by definition that $\alpha$ is a so-called \emph{splitting automorphism of $G$}. The terminology \enquote{splitting automorphism} was introduced by Gorchakov in \cite{Gor65a}, and splitting automorphisms have been studied by various authors. For example, Kegel proved that a finite group with a splitting automorphism of prime order is nilpotent \cite[Satz 1]{Keg61a}, and Ersoy showed that a finite group is solvable if it admits a splitting automorphism of odd order \cite[Theorem 1.4]{Ers16a}.

Finally, we would like to note that there is a quantitative version of Rowley's theorem, proved by Hartley \cite[Theorem A$'$]{Har92a}, which states that for any fixed positive integer $n$, as the order of a nonabelian finite simple group $S$ tends to $\infty$, so does the minimum number of fixed points of an order $n$ automorphism of $S$. Note that the assumption of keeping the automorphism order fixed is essential, since, for example, the standard Frobenius field automorphism
\[
\overline{\begin{pmatrix}a & b \\ c & d\end{pmatrix}}\mapsto\overline{\begin{pmatrix}a^2 & b^2 \\ c^2 & d^2\end{pmatrix}}
\]
of $\PSL_2(2^f)$ has $|\PSL_2(2)|=6$ fixed points for all $f\in\IN^+$. Since the function $G\rightarrow G$, $g\mapsto g^{-1}\alpha(g)$, where $\alpha$ is an automorphism of the group $G$, is constant on the right cosets of the fixed point subgroup of $\alpha$ in $G$, an equivalent reformulation of Hartley's theorem is that for nonabelian finite simple groups $S$, as $|S|\to\infty$, one has that
\[
\frac{1}{|S|}\max_{\alpha\in\Aut(S),\ord(\alpha)=n}{|\{s^{-1}\alpha(s)\mid s\in S\}|}\to0.
\]
This raises the following analogous question, the answer to which is \enquote{no} by known results:

\begin{question}\label{ques4}
Is it true that for each given positive integer $n$, as the order of a nonabelian finite simple group $S$ tends to $\infty$, one has that
\[
\frac{1}{|S|}\max_{\alpha\in\Aut(S),\ord(\alpha)=n}{|\{s\alpha(s)\mid s\in S\}|}\to 0?
\]
\end{question}

Indeed, it follows from \cite[Theorem 1.1(b,c)]{BPS09a} that the proportion of elements of odd order in a finite simple group of Lie type of bounded rank is bounded away from $0$, so that the answer to Question \ref{ques4} is \enquote{no} already for $n=1$.

\section{Acknowledgements}\label{sec5}

The author would like to thank Arne Winterhof for suggesting to work on complete mappings on groups. Moreover, he would like to express his gratitude towards Peter Cameron, Michael Giudici, Laszlo Merai, Cheryl Praeger, Csaba Schneider and Arne Winterhof for some helpful comments during the work on this paper.

\end{document}